\documentclass[12pt]{article}

\usepackage{amsxtra}
\usepackage{amsthm}
\usepackage{amsmath,amscd}
\usepackage{amssymb}
\usepackage{amsfonts}
\usepackage{supertabular}
\usepackage{pstricks}
\usepackage{pst-plot}
\usepackage{pstricks-add}
\usepackage{pst-eps}
\usepackage{makeidx}
\usepackage{bm}

\newcommand{\vect}[1]{\mathbf{#1}}

\newtheorem{theorem}{Theorem}[section]

\newtheorem{definition}[theorem]{Definition}

\newtheorem{remark}[theorem]{Remark}
\newtheorem{corollary}[theorem]{Corollary}

\DeclareMathOperator\SO{SO}

\newcommand{\R}{{\mathbb R}}
\newcommand{\Sph}{{\mathbb S}}

\newcommand{\UH}{{\mathbb H}}

\makeindex

\begin{document}

\title{On the Regularization of the Kepler Problem}

\author
{Gert Heckman and Tim de Laat\\
Radboud University Nijmegen}

\date{Dedicated to the memory of Hans Duistermaat (1942-2010)}

\maketitle

\section{Introduction}

The Kepler problem is an ancient problem, solved for the first time by Newton 
and well studied over more than three centuries \cite{Newton}, \cite{Milnor}, 
\cite{Haandel-Heckman}, \cite{Guillemin-Sternberg}, \cite{Cordani}. 
Newton showed that the solutions of his equation of motion in Hamiltonian form
\[ \dot{\vect{q}}=\vect{p}\;,\;\dot{\vect{p}}=-\vect{q}/q^3 \] 
are planar ellipses traversed according to the area law. 
What is commonly called the Kepler problem is a study of the solutions of these equations, 
preferably from a geometric point of view and taking symmetry arguments into account.
It might come as a surprise that the past half century we have still witnessed new insights on the Kepler problem.

A key point in a modern treatment of the Kepler problem in $\R^n$ is the 
regularization of the collision orbits, which goes back to Moser \cite{Moser}.
The treatment of Moser relates the Kepler flow for a fixed negative energy level 
to the geodesic flow on the sphere $\Sph^n$. The method of Moser is geometrically
well motivated. The main calculation to be done is the determination of the explicit 
transformation formulas for stereographic projection on the level of cotangent bundles.

An alternative approach to the regularization, having the advantage of transforming canonically
the negative energy part of phase space $\R^{2n}$ for the Kepler Hamiltonian
into the punctured cotangent bundle of $\Sph^n$, was found by Ligon and Schaaf \cite{Ligon-Schaaf}.
Unfortunately the article by Ligon and Schaaf requires a good deal of computations.
Cushman and Duistermaat have given a simplified treatment of the Ligon-Schaaf regularization map, 
but their calculations are still more laborious than one would like \cite{Cushman-Duistermaat}.

The main novelty of this article is to show that the Ligon-Schaaf regularization map
is almost trivially understood as an adaptation of the Moser regularization map.
For this reason we recall in Section 2 the Moser regularization, essentially following the original paper by Moser.
The Ligon-Schaaf map is the natural adaptation of the Moser map intertwining the Kepler flow
on the negative energy part $P_-$ of the phase space $\R^{2n}$
and the (geodesic) Delaunay flow on the punctured cotangent bundle $T^{\times}$ of the sphere $\Sph^n$
in a canonical way.

The Kepler problem has hidden symmetry in the sense that the visible symmetry by the orthogonal group 
of size $n$ extends to size $(n+1)$. Of course this hidden symmetry becomes naturally visible on $\Sph^n$,
and all one has to do is to calculate the pull-back under the Ligon-Schaaf map of all angular momenta on $T^{\times}$.
So our discussion in Section 4 of this hidden symmetry is based on the (defined)
equivariance of the Ligon-Schaaf map for the full symmetry group. In this way one is naturally led to
the components of the Lenz vector (divided by the square root of $-2H$) for the extension of Lie algebras
from $\mathfrak{so}(n)$ to $\mathfrak{so}(n+1)$.

All arguments of this paper have natural adaptations from the negative energy
part $P_-$ of phase space to the positive energy part $P_+$ of phase space, in which case
$T^{\times}$ becomes the punctured cotangent bundle of hyperbolic space $\UH^n$.
Details are easily filled in and left to the reader.
At another occasion we hope to discuss the implications of the Ligon-Schaaf regularization 
in classical mechanics for the quantum mechanics of the hydrogen atom \cite{Pauli}, \cite{Fock}, \cite{Bargmann}.

\section{Regularization after Moser}

Let us describe the transformation formulas for stereographic projection.
Let $\vect{n}=(0,\cdots,0,1)$ be the north pole of the unit sphere $\Sph^n$ in $\R^{n+1}$,
and let $\vect{u}=(u_1,\cdots,u_{n+1})$ be another point of $\Sph^n$.
Let $\vect{v}=(v_1,\cdots,v_{n+1})$ be a covector on $\Sph^n$ at $\vect{u}$, so that
\[ \vect{u}\neq\vect{n}\;,\;\vect{u}\cdot\vect{u}=1 \;,\;\vect{u}\cdot\vect{v}=0 \]
are the constraints on $\vect{u},\vect{v}$ in $\R^{n+1}$.
The line through $\vect{n}$ and $\vect{u}$ intersects the hyperplane $\R^n$ orthogonal to $\vect{n}$ in the point $\vect{x}$,
and let $\vect{y}$ be a covector on $\R^n$ at $\vect{x}$. 
The map $\vect{x}\mapsto \vect{u}$ is called stereographic projection,
and $\vect{u}\mapsto\vect{x}$ inverse stereographic projection.

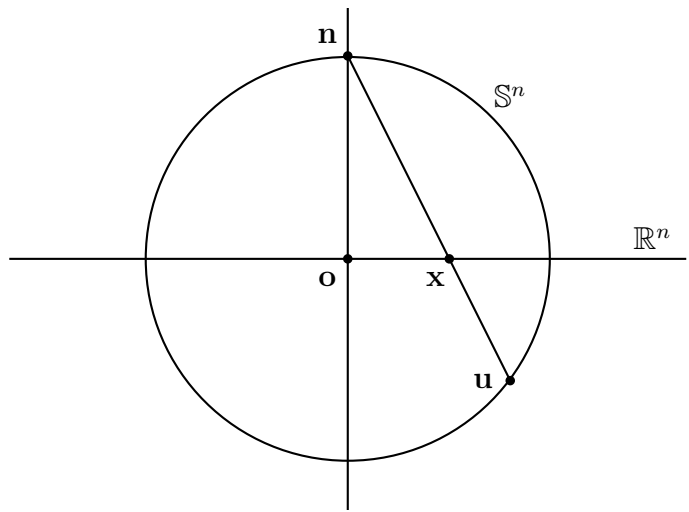
\begin{center}
\psset{unit=0.9mm}
\begin{pspicture}*(-50,-37)(50,37)

\pscircle(0,0){30}

\psline(-50,0)(50,0)
\psline(0,-40)(0,40)
\psline(0,30)(24,-18)

\psdot(0,0)
\psdot(0,30)
\psdot(15,0)
\psdot(24,-18)

\rput(-3,33){$\vect{n}$}
\rput(-3,-3){$\vect{o}$}
\rput(13,-3){$\vect{x}$}
\rput(20,-18){$\vect{u}$}

\rput(24,24){$\Sph^n$}
\rput(45,3){$\R^n$}

\end{pspicture}
\end{center}

\begin{theorem}\label{stereographic projection formulas}
The transformation formulas for the inverse stereographic projection $T^{\ast}\Sph^n \rightarrow T^{\ast}\R^n, (\vect{u},\vect{v})\mapsto(\vect{x},\vect{y})$ are given by
\begin{align*}
   x_k &= u_k/(1-u_{n+1}) \\
   y_k &= v_k(1-u_{n+1})+v_{n+1}u_k
\end{align*}
and for the stereographic projection 
$T^{\ast}\R^n \rightarrow T^{\ast}\Sph^n,(\vect{x},\vect{y})\mapsto(\vect{u},\vect{v})$ become
\begin{align*}
   u_k &= 2x_k/(x^2+1)                            & u_{n+1} &= (x^2-1)/(x^2+1) \\
   v_k &= (x^2+1)y_k/2-(\vect{x}\cdot\vect{y})x_k & v_{n+1} &= \vect{x}\cdot\vect{y}
\end{align*}
for $k=1,\cdots,n$. These transformations are canonical in the sense that the symplectic forms
$\sum_1^n dx_k \wedge dy_k$ and (the restriction of) $\sum_1^{n+1} du_k \wedge dv_k$ match.
\end{theorem}

\begin{proof}
Clearly the last coordinate of $\vect{x}=\lambda\vect{u}+(1-\lambda)\vect{n}$ vanishes, and so $\lambda(1-u_{n+1})=1$.
Conversely $\vect{u}=\mu\vect{x}+(1-\mu)\vect{n}$ has unit length, and so $\mu(x^2+1)=2$. 
Hence the transformation formulas between $\Sph^n$ and $\R^n$ are obvious. 

By rotational symmetry we may suppose that $n=1$. In that case
\[ x_1=u_1/(1-u_2)\;,\;y_1=v_1(1-u_2)+v_2u_1 \]
and the canonical one form on $T^{\ast}\R$ becomes
\begin{align*}
  y_1dx_1 &= (v_1(1-u_2)+v_2u_1)(du_1/(1-u_2)+u_1du_2/(1-u_2)^2) \\
          &= v_1du_1+v_1u_1du_2/(1-u_2)+v_2u_1du_1/(1-u_2)+v_2u_1^2du_2/(1-u_2)^2 \\
          &= v_1du_1+v_2\{-u_2du_2+u_1du_1+(1+u_2)du_2\}/(1-u_2) \\
          &= v_1du_1+v_2du_2+v_2\{u_1du_1+u_2du_2\}/(1-u_2) \\
          &= v_1du_1+v_2du_2
\end{align*}
as should, using $v_1u_1+v_2u_2=0,u_1^2+u_2^2=1,u_1du_1+u_2du_2=0$.

For the inverse mapping again suppose $n=1$. Then we have
\[ y_1x_1=\{v_1(1-u_2)+v_2u_1\}u_1/(1-u_2)=v_1u_1+v_2(1+u_2)=v_2 \]
using $u_1^2+u_2^2=1,v_1u_1+v_2u_2=0$, and therefore
\[ v_1=(y_1-v_2u_1)/(1-u_2)=y_1(x^2+1)/2-(y_1x_1)x_1 \]
as desired, using $(1-u_2)=2/(x^2+1),u_1/(1-u_2)=x_1$.
Since any diffeomorphism of manifolds induces a symplectomorphism of cotangent bundles the theorem follows.
\end{proof}

\begin{corollary}\label{stereographic metric formula}
The invariant metric on $\Sph^n$ transforms to
\[ v^2=(x^2+1)^2y^2/4 \]
under stereographic projection.
\end{corollary}

\begin{proof}
Indeed the relation
\[ v_k=(x^2+1)y_k/2-(\vect{x}\cdot\vect{y})x_k \;,\; v_{n+1}=\vect{x}\cdot\vect{y} \]
for $k=1\cdots,n$ gives 
\[ v^2=(x^2+1)^2y^2/4+(\vect{x}\cdot\vect{y})^2x^2-(x^2+1)(\vect{x}\cdot\vect{y})^2+(\vect{x}\cdot\vect{y})^2 \]
and the formula follows.
\end{proof}

Consider the function $F(\vect{u},\vect{v})=v^2/2$ on the cotangent bundle $T^{\ast}\Sph^n$ of the sphere $\Sph^n$.
The trajectories of the Hamiltonian flow of $F$ on the level hypersurface $F=1/2$ in $T^{\ast}\Sph^n$ 
project onto great circles on $\Sph^n$ traversed in arc length time $s$ with period $2\pi$.
By the above corollary the trajectories of the Hamiltonian flow on $T^{\ast}\R^n$ of the function 
\[ F(\vect{x},\vect{y})=(x^2+1)^2y^2/8 \]
are images under stereographic projection of the trajectories of this geodesic flow on $T^{\ast}\Sph^n$. 
Consider the two related functions
\begin{eqnarray*}
   G(\vect{x},\vect{y})=\sqrt{2F(\vect{x},\vect{y})}-1=(x^2+1)y/2-1 \\
   \hat{H}(\vect{x},\vect{y})=G(\vect{x},\vect{y})/y-1/2=x^2/2-1/y
\end{eqnarray*}
on $T^{\ast}\R^n$.
The Hamiltonian vector fields of $F$ and $G$ on $T^{\ast}\R^n$ coincide on the level hypersurface 
$F=1/2$ (or equivalently $G=0$ or $\hat{H}=-1/2$), since the derivative 
\[ \{r\mapsto\sqrt{2r}-1\}'=\{r\mapsto1/\sqrt{2r}\} \]
is equal to $1$ for $r=1/2$. 
Hence we have proved the following regularization theorem of Moser \cite{Moser}.

\begin{theorem}\label{Moser regularization theorem}
On the level hypersurface $F=1/2$ the trajectories of the Hamiltonian flow of the function 
$F(\vect{u},\vect{v})=v^2/2$ on $T^{\ast}(\Sph^n-\{\vect{n}\})$ traversed in time $s$ equal to arc length
transform under inverse stereographic projection to trajectories of the Hamiltonian flow 
of the function $\hat{H}(\vect{x},\vect{y})=x^2/2-1/y$ traversed in real time $t$
on the level hypersurface $\hat{H}=-1/2$ with 
\[ \frac{ds}{dt}=\frac{1}{y} \]
relating the two time parameters $s$ and $t$.
\end{theorem}

If we set $\vect{x}=\vect{p},\vect{y}=-\vect{q}$ then $(\vect{q},\vect{p})\mapsto(\vect{x},\vect{y})$
is a canonical transformation of $\R^{2n}$, called the geometric Fourier transform, and  
\[ H(\vect{q},\vect{p})=\hat{H}(\vect{x},\vect{y})=p^2/2-1/q \] 
becomes the Kepler Hamiltonian. 

\begin{corollary}\label{Moser map}
The Moser regularization map $\Phi_M: T^{\ast}\R^n \rightarrow T^{\ast}(\Sph^n-\{\vect{n}\})$ is defined
as the composition of stereographic projection with geometric Fourier transform.
It is a symplectomorphism and explicitly given by the formula 
$(\vect{q},\vect{p}) \mapsto \Phi_M(\vect{q},\vect{p})=(\vect{u},\vect{v})$ with
\[ \vect{u}=(2\vect{p}/(p^2+1),2p^2/(p^2+1)-1)\;,\;
   \vect{v}=(-(p^2+1)\vect{q}/2+(\vect{q}\cdot\vect{p})\vect{p},-\vect{q}\cdot\vect{p}) \]
as is clear from Theorem{\;\ref{stereographic projection formulas}}. 
On the level hypersurface $H=-1/2$ the Moser map transforms the Kepler flow with time $t$
to the geodesic flow on $\Sph^n-\{\vect{n}\}$ with arc length time $s$,
with the infinitesimal Kepler equation
\[ \frac{ds}{dt}=\frac{1}{q} \]
relating the two different time parameters.
\end{corollary}

In this corollary the equivalence of the Kepler problem with the geodesic flow of the sphere
is established only on the energy hypersurface $H=-1/2$. The general case of negative energy $H<0$ can be
reduced to this case by a scaling of variables with $\R_+=\{\rho>0\}$ according to
\[ \vect{q}\mapsto\rho^2\vect{q}\;,\;\vect{p}\mapsto\rho^{-1}\vect{p}\;,\;H\mapsto\rho^{-2}H\;,\;t\mapsto\rho^3t \]
and the symplectic form $\omega=\sum dq_k \wedge dp_k$ scales according to $\omega\mapsto\rho\omega$.
This ends our discussion of the Moser regularization.

\section{Regularization after Ligon and Schaaf}

Let us fix some notations. 
The phase space $T^{\ast}\R^n$ of the Kepler problem with 
canonical coordinates $(\vect{q},\vect{p})$ will be denoted by $P$.
Let us denote
\begin{eqnarray*}
P_- &=& \{(\vect{q},\vect{p})\in P;\vect{q}\neq \vect{o},H(\vect{q},\vect{p})<0\} \\
P_{-1/2} &=& \{(\vect{q},\vect{p})\in P;\vect{q}\neq \vect{o},H(\vect{q},\vect{p})=-1/2\}   
\end{eqnarray*}
with $H(\vect{q},\vect{p})=p^2/2-1/q$ the Kepler Hamiltonian as before.
Likewise let the cotangent bundle 
$T^{\ast}\Sph^n=\{(\vect{u},\vect{v})\in T^{\ast}\R^{n+1};u=1,\vect{u}\cdot\vect{v}=0\}$ 
of the sphere will be denoted by $T$, and put
\begin{eqnarray*}
T^{\times} &=& \{(\vect{u},\vect{v})\in T; v \neq 0\} \\
T_-        &=& \{(\vect{u},\vect{v})\in T; \vect{u}\neq\vect{n},v \neq 0\} \\
T_{-1/2}   &=& \{(\vect{u},\vect{v})\in T; \vect{u}\neq\vect{n},v=1\}
\end{eqnarray*}
for the various submanifolds of $T$. 

\begin{definition}\label{Delaunay definition}
The Delaunay Hamiltonian $\tilde{H}$ is defined by
\[ \tilde{H}(\vect{u},\vect{v})=-\frac{1}{2v^2} \]
on the punctured cotangent bundle $T^{\times}$ of $\Sph^n$.
\end{definition}

It is clear that the flow of the Delaunay Hamiltonian has a similar scaling symmetry
\[ \vect{u}\mapsto\vect{u},\vect{v}\mapsto\rho\vect{v},\tilde{H}\mapsto\rho^{-2}\tilde{H},t\mapsto\rho^3t\]
as the Kepler problem, and likewise the symplectic form $\tilde{\omega}=\sum du_k \wedge dv_k$ 
scales according to $\tilde{\omega}\mapsto\rho\tilde{\omega}$.

\begin{remark}
The Moser fibration $\Pi_M: P_- \rightarrow T_{-1/2}$ with $\Pi_M(\vect{q},\vect{p})=(\vect{u},\vect{v})$ 
is defined by
\[ \vect{u}=(\sqrt{-2H}q\vect{p},qp^2-1)\;,\;
   \vect{v}=(-\vect{q}/q+(\vect{q}\cdot\vect{p})\vect{p},-\sqrt{-2H}\vect{q}\cdot\vect{p})\;. \]
Note that $\Pi_M$ is invariant under the action of the scale group $\R_+$ on $P_-$.
On the submanifold $P_{-1/2}$ we have $(p^2+1)/2=1/q$, and therefore
\[ \Phi_M(\vect{q},\vect{p})=\Pi_M(\vect{q},\vect{p}) \]
on $P_{-1/2}$ by a direct comparison with the formula for the Moser map $\Phi_M$ in Corollary{\;\ref{Moser map}}. 
In turn this equality implies that $\Pi_M$ is indeed a smooth fibration with fibers the orbits of the scale group $\R_+$.
\end{remark}

\begin{definition}\label{Ligon-Schaaf definition}
The Ligon-Schaaf regularization map $\Phi_{LS}:P_- \rightarrow T_-$ 
is defined by $\Phi_{LS}(\vect{q},\vect{p})=(\vect{r},\vect{s})$ with
\[ \vect{r}=((\cos v_{n+1})\vect{u}+(\sin v_{n+1})\vect{v}) \;,\;
   \vect{s}=((-\sin v_{n+1})\vect{u}+(\cos v_{n+1})\vect{v})/\sqrt{-2H} \]
with $\vect{u},\vect{v}\in\Sph^n\subset\R^{n+1}$ given by
\[ \vect{u}=(\sqrt{-2H}q\vect{p},qp^2-1)\;,\;
   \vect{v}=(-\vect{q}/q+(\vect{q}\cdot\vect{p})\vect{p},-\sqrt{-2H}(\vect{q}\cdot\vect{p})) \]
as the components of the Moser fibration $\Pi_M(\vect{q},\vect{p})=(\vect{u},\vect{v})$.
\end{definition}

By definition the Ligon-Schaaf map is equivariant for the action of the scale group $\R_+$
on $P_-$ and $T^{\times}$, which in turn is equivalent to $\Phi_{LS}^{\ast}\tilde{H}=H$. 

For $(\vect{u},\vect{v}) \in T_{-1/2}$ the vectors $\vect{u},\vect{v}$ form an orthonormal basis
in the plane $\R\vect{u}+\R\vect{v}$. If we define a complex structure on this plane by
\[ i\vect{u}=\vect{v}\;,\;i\vect{v}=-\vect{u} \]
then the Ligon-Schaaf map can be written in the compact form
\[ \vect{r}=e^{iv_{n+1}}\vect{u}\;,\;\vect{s}=e^{iv_{n+1}}\vect{v}/\sqrt{-2H}\;. \]

We claim that for $(\vect{q},\vect{p}) \in P_{-1/2}$ and $(\vect{u},\vect{v}),(\vect{r},\vect{s}) \in T_{-1/2}$
related as above the following three differential equations
\begin{eqnarray*}
   \frac{d\vect{q}}{dt}=\vect{p}  &,& \frac{d\vect{p}}{dt}=-\vect{q}/q^3 \\
   \frac{d\vect{u}}{ds}=\vect{v} &,& \frac{d\vect{v}}{ds}=-\vect{u} \\
   \frac{d\vect{r}}{dt}=\vect{s} &,& \frac{d\vect{s}}{dt}=-\vect{r}
\end{eqnarray*}
are equivalent with time parameters $t$ and $s$ related by $ds/dt=1/q$. 
The first two of these are equivalent under the Moser map $\Phi_M$ by Corollary{\;\ref{Moser map}}.
The equivalence with the third one follows from
\[ \frac{d\vect{r}}{dt}=i\frac{dv_{n+1}}{dt}\vect{r}+e^{iv_{n+1}}\frac{ds}{dt}\frac{d\vect{u}}{ds}
   =(-p^2+1/q)\vect{s}+\vect{s}/q=(-2H)\vect{s}=\vect{s} \]
and likewise for $\vect{s}=i\vect{r}$. 
This shows that the first and the third equation are equivalent under the Ligon-Schaaf map $\Phi_{LS}$.
Using that the Kepler and the Delaunay Hamiltonians have the same scale symmetry 
and that $\Phi_{LS}$ is equivariant for the two scale symmetries it follows that
$\Phi_{LS}$ intertwines the Kepler flow on $P_-$ and the Delaunay flow on $T_-$.

Since $\Phi_{LS}^{\ast}\tilde{H}=H$ it follows that $\Phi_{LS}:P_- \rightarrow T_-$ is a symplectomorphism.
Indeed $\Phi_{LS}$ is obtained from the canonical Moser map $\Phi_M: P \rightarrow T$
by a canonical modification along the ruled surfaces with base the Kepler orbits in $P_{-1/2}$ and rulings the scale action.
For such a ruled surface $S$ in $P_-$ the form $\omega|_{S}=dt \wedge dH$ is mapped
under $\Phi_{LS}$ to the form $\tilde{\omega}|_{\tilde{S}}=dt \wedge d\tilde{H}$ on $\tilde{S}=\Phi_{LS}(S)$.
This proves the following result of Ligon and Schaaf \cite{Ligon-Schaaf}.

\begin{theorem}\label{Ligon-Schaaf theorem}
The Ligon-Schaaf regularization map $\Phi_{LS}: P_- \rightarrow T_-$ is a symplectomorphism
intertwining the Kepler flow on $P_-$ with the Delaunay flow on $T_-$ (with respect to the same time parameter $t$).
\end{theorem}

The regularization is obtained by the partial compactification $T_- \hookrightarrow T^{\times}$.
The incomplete Kepler flow of the collision orbits on $P_-$ is regularized under $\Phi_{LS}$ 
by $T_- \hookrightarrow T^{\times}$ to the complete Delaunay flow on $T^{\times}$. 
For this reason the punctured cotangent bundle $T^{\times}$ of $\Sph^n$ is called the Kepler manifold \cite{Souriau}.
We now move on to discuss the symmetry of the Kepler problem. 

\section{The symmetry group}

If we write $L_{ij}=q_ip_j-q_jp_i$ for the components of angular momentum $\vect{L}=\vect{q}\wedge\vect{p}$
on the phase space $\R^{2n}$ then $\{L_{ij},L_{jk}\}=L_{ki}$ while $\{L_{ij},L_{kl}\}=0$ if $\#\{i,j,k,l\} \neq 3$.
These are the commutation relations for the Lie algebra $\mathfrak{so}(n)$.

\begin{theorem}
On the negative energy part $P_-$ of phase space put
\[ L_{i(n+1)}=-L_{(n+1)i}=K_i/\sqrt{-2H} \]
with $K_i$ the components of the Lenz vector
\[ \vect{K}=(p^2-1/q)\vect{q}-(\vect{q}\cdot\vect{p})\vect{p}\;. \]
If we write $M_{ij}=r_is_j-r_js_i$ for the components of the angular momentum $\vect{M}=\vect{r}\wedge\vect{s}$
on the phase space $\R^{2(n+1)}$ then
\[ \Phi_{LS}^{\ast}M_{ij}=L_{ij} \]
for $i,j=1,\cdots,n+1$. 
In particular $\{L_{ij},L_{jk}\}=L_{ki}$ and $\{L_{ij},L_{kl}\}=0$ if $\#\{i,j,k,l\} \neq 3$ for $i,j,k,l=1,\cdots,n+1$.
The conclusion is that the Ligon-Schaaf regularization map
\[ \Phi_{LS}:P_-\rightarrow T_-  \]
intertwines the two infinitesimal Hamiltonian actions of $\mathfrak{so}(n+1)$.
\end{theorem}

\begin{proof}
Using the formula for $\Phi_{LS}$ of Definition{\;\ref{Ligon-Schaaf definition}} 
the relation $\Phi_{LS}^{\ast}M_{ij}=L_{ij}$ is a rather straightforward exercise in substition.

Indeed for $i,j=1,\cdots,n$ and $M_{ij}=r_is_j-r_js_i$ the expression for $\Phi_{LS}^{\ast}M_{ij}$ 
at $(\vect{q},\vect{p})$ is equal to
\begin{eqnarray*}
&=& \{(\cos v_{n+1})(\sqrt{-2H}qp_i)+(\sin v_{n+1})(-q_i/q+(\vect{q}\cdot\vect{p})p_i)\} \\
& & \times \{(-\sin v_{n+1})(\sqrt{-2H}qp_j)+(\cos v_{n+1})(-q_j/q+(\vect{q}\cdot\vect{p})p_j)\}/\sqrt{-2H} \\
& & -\{i \leftrightarrow j \} \\
&=& \{(\cos  v_{n+1}\sin v_{n+1})(-q^2p_ip_j)\}\sqrt{-2H} \\
& & +\{(\cos^2 v_{n+1})(-p_iq_j+q(\vect{q}\cdot\vect{p})p_ip_j)+(\sin^2 v_{n+1})(q_ip_j-q(\vect{q}\cdot\vect{p})p_ip_j)\} \\
& & +\{(\cos  v_{n+1}\sin v_{n+1})(-q_i/q+(\vect{q}\cdot\vect{p})p_i)(-q_j/q+(\vect{q}\cdot\vect{p})p_j)\}/\sqrt{-2H} \\
& & - \{ i \leftrightarrow j \} \\
&=& \{\cos^2 v_{n+1}(q_ip_j-q_jp_i)+\sin^2 v_{n+1}(q_ip_j-q_jp_i)\}=(q_ip_j-q_jp_i)=L_{ij}
\end{eqnarray*}
with $\{i \leftrightarrow j\}$ denoting the same expression with $i$ and $j$ interchanged.
Of course this formula is also obvious because our geometric construction of the
Ligon-Schaaf regularization map is equivariant for the action of $\mathfrak{so}(n)$.

Likewise for $i=1,\cdots,n$ and $M_{i(n+1)}=r_is_{n+1}-r_{n+1}s_i$ the expression 
for $\Phi_{LS}^{\ast}M_{i(n+1)}$ at $(\vect{q},\vect{p})$ is equal to
\begin{eqnarray*}
&=& \{(\cos v_{n+1})(\sqrt{-2H}qp_i)+(\sin v_{n+1})(-q_i/q+(\vect{q}\cdot\vect{p})p_i)\} \\
& & \times \{(-\sin v_{n+1})(qp^2-1)+(\cos v_{n+1})(-\sqrt{-2H}(\vect{q}\cdot\vect{p}))\}/\sqrt{-2H} \\
& & -\{(\cos v_{n+1})(qp^2-1)+(\sin v_{n+1})(-\sqrt{-2H}(\vect{q}\cdot\vect{p}))\} \\
& & \times\{(-\sin v_{n+1})(\sqrt{-2H}qp_i)+(\cos v_{n+1})(-q_i/q+(\vect{q}\cdot\vect{p})p_i)\}/\sqrt{-2H} \\
&=& \{(-\cos v_{n+1}\sin v_{n+1}+\cos v_{n+1}\sin v_{n+1})(qp^2-1)qp_i\} \\
& & +\{(-\sin v_{n+1}\cos v_{n+1}+\sin v_{n+1}\cos v_{n+1})(\vect{q}\cdot\vect{p})(-q_i/q+(\vect{q}\cdot\vect{p})p_i)\} \\
& & +\{(\cos^2 v_{n+1}+\sin^2 v_{n+1})(2Hq(\vect{q}\cdot\vect{p})p_i)\}/\sqrt{-2H} \\
& & +\{(-\sin^2 v_{n+1}-\cos^2 v_{n+1})(qp^2-1)(-q_i/q+(\vect{q}\cdot\vect{p})p_i)\}/\sqrt{-2H} \\
&=& \{2Hq(\vect{q}\cdot\vect{p})p_i-(qp^2-1)(-q_i/q+(\vect{q}\cdot\vect{p})p_i)\}/\sqrt{-2H} \\
&=& \{(qp^2-2)(\vect{q}\cdot\vect{p})p_i+(qp^2-1)q_i/q-(qp^2-1)(\vect{q}\cdot\vect{p})p_i\}/\sqrt{-2H} \\
&=& \{(p^2-1/q)q_i-(\vect{q}\cdot\vect{p})p_i\}/\sqrt{-2H}=K_i/\sqrt{-2H}=L_{i(n+1)}
\end{eqnarray*}
which proves the theorem.
\end{proof}

\begin{corollary}
On $P_-$ the Hamiltonian vector fields of the functions $L_{ij}$ for $i,j=1,\cdots,n+1$
integrate to an incomplete Hamiltonian action of the rotation group $\SO(n+1)$.
Moreover the Ligon-Schaaf map 
\[ \Phi_{LS}:P_- \rightarrow T_- \] 
intertwines this action with the standard complete Hamiltonian action
of $\SO(n+1)$ on the punctured cotangent bundle $T^{\times}$ of $\Sph^n$.
In particular $\Phi_{LS}$ intertwines the corresponding moment maps 
\[ \mu:P_- \rightarrow \mathfrak{so}(n+1)^{\ast}\;,\; \tilde{\mu}:T^{\times} \rightarrow \mathfrak{so}(n+1)^{\ast} \]
in the sense that $\mu(\vect{q},\vect{p})=\tilde{\mu}(\Phi_{LS}(\vect{q},\vect{p}))$. In turn this implies that
\[ \mu^2=L^2+K^2/(-2H)=1/(-2H) \]
because $\tilde{\mu}^2=1/(-2\tilde{H})$.
\end{corollary}

The incomplete Hamiltonian action of $\mathfrak{so}(n+1)$ on $P_-$ is regularized under $\Phi_{LS}$ by the partial
compactification $T_- \hookrightarrow T^{\times}$ to a complete Hamiltonian action of $\mathfrak{so}(n+1)$ on $T^{\times}$.
The fact that $\Phi_{LS}$ intertwines the two moment maps was the point of departure 
of Cushman and Duistermaat \cite{Cushman-Duistermaat}.
The next corollary is due to Gy\"orgyi \cite{Gyorgyi}, but on the more refined quantum level
it was obtained before by Pauli in his spectral analysis of the hydrogen atom \cite{Pauli}.

\begin{corollary}
On the phase space $\{(\vect{q},\vect{p})\in P;\vect{q}\neq\vect{o}\}$ of $\R^n-\{\vect{o}\}$ 
the components $K_i$ of the Lenz vector satisfy the Poisson bracket relations
\[ \{L_{ij},K_k\}=\delta_{ik}K_j-\delta_{jk}K_i \;,\; \{K_i,K_j\}=-2HL_{ij} \]
for $i,j,k=1,\cdots,n$. 
\end{corollary}

\begin{proof}
These Poisson brackets hold on $P_-$ by the above theorem, and extend from $P_-$ to 
$\{(\vect{q},\vect{p})\in P;\vect{q}\neq\vect{o}\}$ as analytic identities.
\end{proof}

In turn this implies that on the positive energy part $P_+$ of phase space the symmetry algebra 
becomes the Lorentz algebra $\mathfrak{so}(n,1)$ rather than the orthogonal algebra $\mathfrak{so}(n+1)$.

\noindent
Gert Heckman, Radboud University Nijmegen, P.O. Box 9010,
\newline 
6500 GL Nijmegen, The Netherlands (E-mail: g.heckman@math.ru.nl)
\newline
Tim de Laat, Radboud University Nijmegen, P.O. Box 9010,
\newline 
6500 GL Nijmegen, The Netherlands (E-mail: tlaat@science.ru.nl)

\end{document}